\documentclass[letterpaper, 10 pt, conference]{ieeeconf}  

\IEEEoverridecommandlockouts                              
\usepackage{graphicx}
\usepackage{amsmath,amssymb}
\usepackage{mathtools}
\usepackage{subcaption}
\usepackage{color}
\usepackage{cite}
\usepackage[bb=dsserif]{mathalpha} 

\newcommand{\norm}[1]{\lVert #1 \rVert}                                 
\newcommand{\inX}[2]{\in\mathbb{#1}^{#2}}                               
\newcommand{\IqZ}{\begin{bmatrix}I_q\\Z\end{bmatrix}}                   
\newcommand{\sqmat}[4]{\begin{bmatrix}#1&#2\\#3&#4\end{bmatrix}}        
\newcommand{\matvec}[2]{\begin{bmatrix}#1\\#2\end{bmatrix}}             
\newcommand{\matvecvec}[3]{\begin{bmatrix}#1\\#2\\#3\end{bmatrix}}      
\newcommand{\Zqr}{\mathcal{Z}^q_r}                                      
\newcommand{\im}{\mathrm{im}\,}                                           
\newcommand{\Zab}[2]{\mathcal{Z}^{#1}_{#2}}                             
\newcommand{\IAB}{\matvecvec{I_n}{A^\top }{B^\top }}                    
\newcommand{\hDel}{\hat{\Delta}}                                        
\newcommand{\hPhi}{\hat{\Phi}}                                          %
\newcommand{\st}{\ \mathrm{s.t.}\ }                                     %

\newtheorem{theorem}{Theorem}
\newtheorem{lemma}[theorem]{Lemma}
\newtheorem{proposition}[theorem]{Proposition}

\newtheorem{assumption}[theorem]{Assumption}
\newtheorem{definition}[theorem]{Definition}
\newtheorem{remark}[theorem]{Remark}

\renewcommand{\QED}{\QEDopen}

\title{\LARGE \bf
Data Informativity for Quadratic Stabilization under Data Perturbation
}

\author{Taira Kaminaga and Hampei Sasahara
\thanks{This work was supported by JSPS KAKENHI Grant Number 24K17296.}
\thanks{T. Kaminaga and H. Sasahara are with Department of Systems and Control Engineering, Graduate School of Engineering,
Institute of Science Tokyo, Tokyo, Japan {\tt\small kaminaga@cyb.sc.e.titech.ac.jp, sasahara@sc.e.titech.ac.jp}}
}

\begin{document}

\maketitle
\thispagestyle{empty}
\pagestyle{empty}


\begin{abstract}
   Assessing data informativity, determining whether the measured data contains sufficient information for a specific control objective, is a fundamental challenge in data-driven control. 
   In noisy scenarios, existing studies deal with system noise and measurement noise separately, using quadratic matrix inequalities. 
   Moreover, the analysis of measurement noise requires restrictive assumptions on noise properties. 
   To provide a unified framework without any restrictions, this study introduces data perturbation, a novel notion that encompasses both existing noise models. 
   It is observed that the admissible system set with data perturbation does not meet preconditions necessary for applying the key lemma in the matrix S-procedure. 
   Our analysis overcomes this limitation by developing an extended version of this lemma, making it applicable to data perturbation. 
   Our results unify the existing analyses while eliminating the need for restrictive assumptions made in the measurement noise scenario.
\end{abstract}

\section{INTRODUCTION}
Data-driven control is a controller synthesis approach that relies on directly using data from the system to be controlled, 
rather than relying on detailed mathematical models obtained through system identification~\cite{DDC:Hou2013,DDC:Markovsky2023_Cont_sys_mag}. 
A fundamental challenge in data-driven control is assessing \emph{data informativity}~\cite{DDC:Waarde2020_TAC_Dinfo}, 
which refers to the extent to which the measured data contains sufficient information for the system analysis and controller synthesis of interest. 
Computationally tractable conditions, typically formulated by linear matrix inequalities (LMIs), 
have been successfully derived for key problems, such as controllability verification, stabilization, and optimal control, in the noise-free case~\cite{DDC:Waarde2023_Cont_sys_mag_informativity}. 
The data informativity framework has also been extended to scenarios where the collected data is corrupted by noise. 
The works~\cite{DDCQMI:Waarde2022_TAC_origin,DDCQMI:Waarde2023_siam_qmi,DDCQMI:BISOFFI2022_Petersen,DDCQMI:BISOFFI2021}
study data informativity under system noise, where the system is governed not only by control input but also by exogeneous disturbances. 
Additionally, the studies~\cite{DDCQMI:BISOFFI2024_CSL,DDC:Miller2024_TAC}
address the errors in variables (EIV) model, which accounts for measurement noise.

Technically, most of those existing studies for noisy data rely on the \emph{matrix S-procedure} using a noise model characterized by \emph{quadratic matrix inequalities (QMIs).} 
QMI, an inequality involving a matrix-valued quadratic function, can capture various noise characteristics, such as its energy bound. 
The core idea of their analysis is as follows: First, the admissible system set, whose elements are consistent with the collected data and the noise model, is described as a set constrained by a QMI. 
Next, the system set whose elements satisfy the property of interest is also represented as a set characterized by another QMI. 
Finally, leveraging the matrix S-lemma or its variants~\cite{DDCQMI:Waarde2022_TAC_origin,DDCQMI:Waarde2023_siam_qmi,DDCQMI:BISOFFI2022_Petersen}, 
a nonconservative LMI condition is derived to characterize their inclusion relationship, which is equivalent to the data informativity.

In this paper, we propose a general framework that encompasses system noise and EIV settings by introducing \emph{data perturbation} as a novel noise model.
Data perturbation refers to noises added to whole input and state data and constrained within a linear subspace. 
We focus on data informativity for quadratic stabilization of an unknown discrete-time linear time-invariant (LTI) system with data perturbation constrained by a QMI. 
We basically follow the existing analysis 
by characterizing the admissible system set by a QMI 
and investigating the inclusion relationship with the stabilizable system set based on the matrix S-lemma. 
However, we observe that the admissible system set with data perturbation does not satisfy preconditions, 
such as convexity, required for applying the existing matrix S-lemma, 
even if the noise set satisfies them. 
This issue does not occur in the system noise setting~\cite{DDCQMI:Waarde2022_TAC_origin,DDCQMI:Waarde2023_siam_qmi} 
and is avoided in the existing analysis of the EIV setting~\cite{DDCQMI:BISOFFI2024_CSL} 
where the signal-to-noise ratio (SNR) is assumed to be sufficiently large. 
We overcome this problem by deriving an extended matrix S-lemma that eliminates the need for the geometric properties of the admissible system set.
The extended matrix S-lemma provides an LMI condition equivalent to the data informativity for quadratic stabilization of all systems consistent with data.

Our contributions are twofold. 
First, we provide data informativity analysis for quadratic stabilization under data perturbation, 
offering a unified version of the existing ones under system noise~\cite{DDCQMI:Waarde2022_TAC_origin,DDCQMI:Waarde2023_siam_qmi} 
and EIV~\cite{DDCQMI:BISOFFI2024_CSL}. 
Second, our analysis eliminates the restrictive assumptions in~\cite{DDCQMI:BISOFFI2024_CSL}, 
which require the data magnitude to be sufficiently large relative to noise and the noise property to be energy bound, a special case of QMI constraints. 
In contrast, our approach accommodates low SNR scenarios and can handle any QMI constraint on the noise.

This paper is organized as follows. 
Sec.~\ref{sec:pre} reviews basic properties of QMI and previous studies on data informativity analysis under system noise and EIV.
Sec.~\ref{sec:main} formulates the problem by introducing data perturbation and deriving the admissible system set.
Subsequently, we derive an extended matrix S-lemma for providing data informativity analysis for quadratic stabilization under data perturbation.
Finally, Sec.~\ref{sec:conclusion} draws conclusions of this paper.

\subsection*{Notation}
We denote 
the set of $n$-dimensional symmetric matrices by $\mathcal{S}^n$,
the $n$-dimensional identity matrix by $I_n$,
the transpose of a matrix $M$ by $M^\top$, 
the pseudo-inverse matrix of a matrix $M$ by $M^\dagger$,
the positive and negative (semi) definiteness of a symmetric matrix $M$ by $M\succ(\succeq)0$ and $M\prec(\preceq)0$ respectively,
and the generalized Schur complement of $D$ in {\small$M=\sqmat{A}{B}{C}{D}$} by $M|D\coloneqq A-BD^\dagger C$.


\section{PRELIMINARIES}\label{sec:pre}
\subsection{Properties of Quadratic Matrix Inequality}\label{subsec:qmi}
We summarize important properties related to quadratic matrix inequalities (QMIs).
A more detailed discussion can be found in~\cite{DDCQMI:Waarde2023_siam_qmi}.
A QMI of the matrix $Z\inX{R}{r\times q}$ is defined as
\begin{equation}\label{qmi:dfn}
   \IqZ^\top  N\IqZ\succeq 0,
\end{equation}
with $N\inX{S}{q+r}$.
A set of the matrices $Z$ satisfying \eqref{qmi:dfn} is defined as $\Zqr(N)\coloneqq\{Z\inX{R}{r\times q}|\eqref{qmi:dfn}\}$.
We denote the submatrices of $N$ as {\small$N=\sqmat{N_{11}}{N_{12}}{N_{21} }{N_{22}}$}
such that $N_{11}\inX{S}{q}$ and $N_{22}\inX{S}{r}$.

First, we introduce the term \emph{matrix ellipsoid}, which characterizes geometric properties of QMIs.
We say a set $\Zqr(N)$ to be a matrix ellipsoid 
when the matrix $N\inX{S}{q+r}$ in the QMI \eqref{qmi:dfn} satisfies the conditions
\begin{equation}\label{qmi:N_condition}
   N_{22}\preceq 0,\quad \ker N_{22}\subseteq\ker N_{12},\quad N|N_{22}\succeq 0.
\end{equation}
When $\ker N_{22}\subseteq\ker N_{12}$ holds, the equation
\begin{equation}\label{N_bunkai}
   N=\sqmat{I_q}{N_{12}N_{22}^\dagger}{0}{I_r}\sqmat{N|N_{22}}{0}{0}{N_{22}}\sqmat{I_q}{N_{12}N_{22}^\dagger}{0}{I_r}^\top
\end{equation}
also holds~\cite[Fact 6.5.4]{bernstein2009matrix}, and then, we have that
\begin{equation}\label{qmi:tenkai}
   \IqZ^\top\! N\!\IqZ\!=\!N|N_{22}\!+\!(Z\!+\!N_{22}^\dagger N_{21})^\top\! N_{22}(Z\!+\!N_{22}^\dagger N_{21}).
\end{equation}
Thus, when \eqref{qmi:N_condition} holds, the QMI \eqref{qmi:dfn} can be written as
\begin{equation}\label{qmi:Z_ellip}
   Q^\top Q-(Z-Z_c)^\top R^\top R(Z-Z_c)\succeq 0,
\end{equation}
where the matrices $Q\inX{R}{q\times q}$ and $R\inX{R}{r\times r}$ are the Cholesky decomposition of the matrices $N|N_{22}$ and $-N_{22}$
such that $Q^\top Q=N|N_{22}$ and $R^\top R=-N_{22}$ respectively,
and $Z_c=-N_{22}^\dagger N_{21}\inX{R}{r\times q}$.
The description \eqref{qmi:Z_ellip} is an extension of the standard description of ellipsoids~\cite[Eq. (3.9)]{boyd1994linear},
and implies that the set $\Zqr(N)$ is convex and symmetric with respect to the point $Z_c$.
We denote a set of the symmetric matrices satisfying \eqref{qmi:N_condition} as
\begin{equation*}
   \Pi_{q, r}\coloneqq \left\{\sqmat{N_{11}}{N_{12}}{N_{21} }{N_{22}}\inX{S}{q+r}\middle|\eqref{qmi:N_condition}\right\}.
\end{equation*}

In this paper, we treat the conditions introduced in~\cite[Sec.~3]{DDCQMI:Waarde2023_siam_qmi} as the definition of the matrix ellipsoid for consistency with the required preconditions in the existing matrix S-lemma (Proposition~\ref{prop:Slem_beta}), although the term has originally been introduced in~\cite[Sec.~2.2]{DDCQMI:BISOFFI2021} where the conditions are given by $N_{22}\prec 0$ and $N|N_{22}\succ 0$.
The original condition is stricter than \eqref{qmi:N_condition} and requires the set $\Zqr(N)$ to be bounded with a nonempty interior.
Note that the matrix ellipsoid defined with \eqref{qmi:N_condition} does not necessarily ensure the boundedness or the existence of nonempty interiors.

Next, we introduce the (strict) matrix S-lemma, proposed in~\cite{DDCQMI:Waarde2022_TAC_origin,DDCQMI:Waarde2023_siam_qmi}.
The standard S-lemma provides an LMI condition under which a quadratic inequality is a consequence of another one~\cite{MATH:Polik2007}.
The matrix S-lemma has originally been derived in~\cite{DDCQMI:Waarde2022_TAC_origin},
and the assumption of the matrix S-lemma is relaxed in~\cite{DDCQMI:Waarde2023_siam_qmi}.
\begin{proposition}[{Matrix S-lemma \cite[Corollary 4.13]{DDCQMI:Waarde2023_siam_qmi}}]\label{prop:Slem_beta}
   Let $M, N\inX{S}{q+r}$.
   Assume that $N\in\Pi_{q, r}$ and $M_{22}\preceq 0$.
   Then,
   \begin{equation}\label{slem_obj}
     \IqZ^\top N\IqZ\succeq 0\Rightarrow\IqZ^\top M\IqZ\succ 0
   \end{equation}
   if and only if there exist scalars $\alpha\geq 0$ and $\beta>0$ such that
   \begin{equation}\label{slem_lmi}
     M-\alpha N\succeq \sqmat{\beta I_q}{0}{0}{0}.
   \end{equation}
\end{proposition}
In the application to control for discrete-time LTI systems, the matrix $Z$ corresponds to system matrices.
Condition \eqref{slem_obj} means that
all system matrices in $\Zqr(N)$ satisfy QMI with the matrix $M$.
Proposition~\ref{prop:Slem_beta} transforms \eqref{slem_obj} into the tractable LMI condition \eqref{slem_lmi}.

\subsection{Data Informativity with System Noise}\label{subsec:system_noise}
We review the data informativity analysis for stabilization under system noise in~\cite{DDCQMI:Waarde2022_TAC_origin,DDCQMI:Waarde2023_siam_qmi}.
Consider the discrete-time LTI system $x(t+1)=A_sx(t)+B_su(t)+w(t)$,
where $x(t) \in \mathbb{R}^n$ is the system state, 
$u(t) \in \mathbb{R}^m$ is the control input, and $w(t) \in \mathbb{R}^n$ is the system noise.
We assume the true system $(A_s, B_s)$ to be unknown.
By open-loop experiment for $t=0,\ldots,T$, we obtain time-series data $X_+=[x(1)\ \cdots\ x(T)]\inX{R}{n\times T},$ $X_-=[x(0)\ \cdots\ x(T-1)]\inX{R}{n\times T},$ and $U_-=[u(0)\ \cdots\ u(T-1)]\inX{R}{m\times T}.$
We denote the unknown system noise sequence in the experiment as $W_-=[w(0)\ \cdots\ w(T-1)]\inX{R}{n\times T}$.
Then, the data and the noise satisfy
\begin{equation}\label{dist:stateeq_mat}
   X_+=A_sX_-+B_sU_-+W_-.
\end{equation}
Additionally, we assume the system noise satisfies an QMI constraint
\begin{equation}\label{dist:qmi_noise}
   \matvec{I_n}{W_-^\top}^\top \Phi^W\matvec{I_n}{W_-^\top}\geq 0,
\end{equation}
i.e. $W_-^\top\in\Zab{n}{T}(\Phi^W)$, 
where $\Phi^W$ is a known matrix satisfying $\Phi^W\in\Pi_{n, T}$.
QMI \eqref{dist:qmi_noise} can describe various constraints on system noise $W_-$, such as energy bound and sample covariance matrix bound~\cite{DDCQMI:Waarde2023_siam_qmi}.
We consider stabilization of the unknown system $(A_s, B_s)$ with a state-feedback controller $u=Kx$ 
by utilizing the prior information on the system noise $\Phi^W$ and the data $(X_+, X_-, U_-)$.

We say that a pair of system matrices $(A,B)$ is consistent with the data $(X_+,X_-,U_-)$ when
there exists a noise sequence $W_-$ such that 
\begin{equation}\label{dist:stateeq_mat2}
   X_+=AX_-+BU_-+W_-.
\end{equation} and \eqref{dist:qmi_noise}.
We define a set of the systems consistent with the data as
\begin{equation}\label{dist:Sigma}
   \Sigma\coloneqq\left\{(A, B)\middle| \exists W_-\inX{R}{n\times T} \st \eqref{dist:qmi_noise}, \eqref{dist:stateeq_mat2}\right\}.
\end{equation}
Note that $(A_s,B_s)\in\Sigma$.
Since the system noise $W_-$ is unknown, the true system $(A_s,B_s)$ cannot be distinguished from other systems in $\Sigma$.
Therefore, we seek a controller $K\inX{R}{m\times n}$ that stabilizes any system in $\Sigma$.
In particular, we consider quadratic stabilization of $\Sigma$, i.e.
every closed-loop system matrix $A+BK$ with $(A,B)\in\Sigma$ has a common quadratic Lyapunov function.
Data informativity for quadratic stabilization is defined as follows.
\begin{definition}\label{dfn:dinfo_qstab}
   The data $(X_+, X_-, U_-)$ is informative for quadratic stabilization 
   if there exist a feedback gain $K\inX{R}{m\times n}$ and a positive definite matrix $P\succ 0$ such that
   \begin{equation}\label{dist:qstab}
      P-(A+BK)P(A+BK)^\top\succ 0
   \end{equation}
   for all $(A,B)\in\Sigma$.
\end{definition}

The analysis begins with characterization of the admissible system set $\Sigma$ by a QMI.
\begin{proposition}[{\cite[Lemma 4]{DDCQMI:Waarde2022_TAC_origin}}]\label{prop:dist_qmi_noise2sys}
   Let $\Phi^W\inX{S}{n+T}$ and $\Sigma$ be defined as \eqref{dist:Sigma}.
   Then, $(A,\ B)\in\Sigma$ is satisfied if and only if
   \begin{equation}\label{dist:qmi_sys}
     \IAB^\top N^W\IAB\succeq 0,
   \end{equation}
   i.e., $[A\ B]^\top\in\Zab{n}{n+m}(N^W)$,
   where $N^W\inX{S}{2n+m}$ is defined as
   \begin{equation}\label{dist:dfn_N}
     N^W\coloneqq\begin{bmatrix}
       I_n&X_+\\0&-X_-\\0&-U_-
     \end{bmatrix}\Phi^W\begin{bmatrix}
       I_n&X_+\\0&-X_-\\0&-U_-
     \end{bmatrix}^\top.
   \end{equation}
\end{proposition}

Proposition \ref{prop:dist_qmi_noise2sys} is obtained by a simple substitution.
The equation \eqref{dist:stateeq_mat2} can be written as 
\begin{equation}\label{dist:noise_description}
   W_-=X_+-AX_--BU_-.
\end{equation}
Substituting \eqref{dist:noise_description} into the QMI on noise \eqref{dist:qmi_noise} derives the QMI on system matrices \eqref{dist:qmi_sys}.

We introduce a proposition that states that when the set of the system noise $\Zab{n}{T}(\Phi^W)$ is a matrix ellipsoid the admissible system set is also a matrix ellipsoid.
\begin{proposition}[{\cite[Section 5.1]{DDCQMI:Waarde2023_siam_qmi}}]\label{prop:dist_ellip}
   A matrix $N^W$ is defined as \eqref{dist:dfn_N}.
   Then, we have that $N^W\in\Pi_{n,n+m}$ if $\Phi^W\in\Pi_{n, T}$.
   In other words, $\Zab{n}{n+m}(N^W)$ is a matrix ellipsoid if $\Zab{n}{T}(\Phi^W)$ is a matrix ellipsoid.
\end{proposition}
The linear relation between the noise and the system matrices explains the inheritance of the geometric properties from the noise set $\Zab{n}{T}(\Phi^W)$ to the system set $\Zab{n}{n+m}(N^W)$.
For example, when \eqref{dist:noise_description} holds and a set of the system noise is convex, a set of the system matrices is also convex.
The proof of Proposition~\ref{prop:dist_ellip} can be found in~\cite[Sec. 5.1]{DDCQMI:Waarde2023_siam_qmi}.

Related to quadratic stabilization, Lyapunov inequality \eqref{dist:qstab} is equivalent to a strict QMI
\begin{equation}\label{dist:qmi_qstab}
   \IAB^\top M\IAB\succ 0,
\end{equation}
with
\begin{equation}\label{dist:dfn_M}
   M\coloneqq\sqmat{P}{0}{0}{-\matvec{I_n}{K}P\matvec{I_n}{K}^\top }.
\end{equation}
Therefore, quadratic stabilization is achieved if and only if all systems that satisfy the QMI \eqref{dist:qmi_sys} also satisfy the QMI \eqref{dist:qmi_qstab}.
We can apply Proposition~\ref{prop:Slem_beta} for the relation between these two QMIs because $\Zab{n}{n+m}(N^W)$ is a matrix ellipsoid.
The matrices $M$ and $N^W$ fulfill the preconditions of Proposition~\ref{prop:Slem_beta} 
because {\small$M_{22}=-\matvec{I_n}{K}P\matvec{I_n}{K}^\top\preceq 0$} and $N^W\in\Pi_{n,n+m}$ owing to Proposition~\ref{prop:dist_ellip}.
Moreover, based on change of variables and Schur complement, we obtain an LMI condition equivalent to the data informativity for quadratic stabilization.
\begin{proposition}[{\cite[Theorem 5.1(a)]{DDCQMI:Waarde2023_siam_qmi}}]\label{prop:dist_result}
   Let $\Phi^W\in\Pi_{q,r}$.
   Assume that the data $(X_+,X_-,U_-)$ and the system noise $W_-$ satisfy \eqref{dist:stateeq_mat} and the system noise satisfies the QMI constraint \eqref{dist:qmi_noise}.
   Then, the data $(X_+,X_-,U_-)$ is informative for quadratic stabilization
   if and only if there exists an $n$-dimensional positive definite matrix $P\succ 0$, a matrix $L\inX{R}{m\times n}$, and a positive scalar $\beta>0$ such that
   \begin{equation}\label{dist:LMI_result}
     \begin{bmatrix}
       P-\beta I_n&0&0&0\\
       0&-P&-L^\top &0\\
       0&-L&0&L\\
       0&0&L^\top &P
     \end{bmatrix}-\sqmat{N^W}{0}{0}{0}\succeq 0,
   \end{equation}
   where $N^W$ is defined as \eqref{dist:dfn_N}.
   If \eqref{dist:LMI_result} is feasible, the controller $K=LP^{-1}\inX{R}{m\times n}$ stabilizes all systems in $\Sigma$.
\end{proposition}

\begin{remark}
   Previous studies also address the situation where system noise is restricted in a subspace~\cite[Rem. 2]{DDCQMI:Waarde2022_TAC_origin}, \cite[Sec. 5.1]{DDCQMI:Waarde2023_siam_qmi}.
   We assume that the system noise $W_-$ can be represented by $W_-=E\hat{W}_-$ with $E\inX{R}{n\times p}$ and $\hat{W}_-\inX{R}{p\times T}$ where $\hat{W}_-$ satisfies the QMI constraint $\hat{W}_-^\top \in\Zab{p}{T}(\hPhi^W)$
   with $\hPhi^W\in\Pi_{p,T}$.
   We additionally assume that $\hPhi_{22}\prec 0$.
   In the above situation, we can also obtain the condition equivalent to the data informativity.
   Indeed, by defining $\Phi^W$ as {\small$\Phi^W\coloneqq\sqmat{E}{0}{0}{I_T}\hPhi^W\sqmat{E}{0}{0}{I_T}^\top$}, 
   $W_-$ also satisfies~\eqref{dist:qmi_noise} and Proposition~\ref{prop:dist_result} can be applied.
\end{remark}


\subsection{Data informativity with EIV}\label{subsec:eiv}
EIV refers to measurement noise on whole input and state data added when collecting them.
The work~\cite{DDCQMI:BISOFFI2024_CSL} has addressed the quadratic stabilization of a discrete-time LTI system with EIV.
Consider an LTI system $x^*(t+1)=A_sx^*(t)+B_su^*(t)$ and the measured data is corrupted by additive noise as
\begin{align*}
   &u(t)=u^*(t)+\delta_u(t),\\
   &x(t)=x^*(t)+\delta_x(t),   
\end{align*}
where $u^*(t)\inX{R}{m}$ and $x^*(t)\inX{R}{n}$ are unknown inputs and states signal, 
$u(t)\inX{R}{m}$ and $x(t)\inX{R}{n}$ are observed inputs and states signal,
and $\delta_u(t)\inX{R}{m}$ and $\delta_x(t)\inX{R}{n}$ are input and measurement noise.
We assume the true system $(A_s,B_s)$ to be unknown.
From this system, we obtain data $(X_+,X_-,U_-)$ satisfying relationship
\begin{equation}\label{eiv:stateeq_mat}
   X_+-\Delta_Z=A_s(X_--\Delta_X)+B_s(U_--\Delta_U),
\end{equation}
where the data $(X_+,X_-,U_-)$ is defined as in Sec.~\ref{subsec:system_noise}
and noise $(\Delta_Z,\Delta_X,\Delta_U)$ is defined as
$\Delta_Z \coloneqq[\delta_x(1)\ \cdots\ \delta_x(T)]\inX{R}{n\times T},$ $\Delta_X \coloneqq[\delta_x(0)\ \cdots\ \delta_x(T-1)]\inX{R}{n\times T},$ $\Delta_U \coloneqq[\delta_u(0)\ \cdots\ \delta_u(T-1)]\inX{R}{m\times T}.$
Additionally, the noise $\Delta\coloneqq [\Delta_Z^\top\ -\Delta_X^\top\ -\Delta_U^\top]^\top$ is assumed to satisfy the energy bound
\begin{equation}\label{eiv:noise_qmi1}
   \Delta\Delta^\top\preceq \Theta
\end{equation}
with a matrix $\Theta\succeq 0$.
The energy bound \eqref{eiv:noise_qmi1} can be described by a QMI constraint
\begin{equation}\label{eiv:noise_qmi2}
   \matvec{I_{2n+m}}{\Delta^\top}^\top\sqmat{\Theta}{0}{0}{-I_T}\matvec{I_{2n+m}}{\Delta^\top}\succeq 0.
\end{equation}
Note that {\small$\sqmat{\Theta}{0}{0}{-I_T}\in\Pi_{2n+m,T}$}.
Moreover, it is assumed in~\cite{DDCQMI:BISOFFI2024_CSL} that
\begin{equation}\label{eiv:assumption}
   \matvec{X_-}{U_-}\matvec{X_-}{U_-}^\top-\Theta_{22}\succ 0,
\end{equation}
where $\Theta_{22}\inX{S}{n+m}$ is the submatrix of {\small$\Theta=\sqmat{\Theta_{11}}{\Theta_{12}}{\Theta_{21}}{\Theta_{22}}$}.
Assumption \eqref{eiv:assumption} means large SNR because
\begin{equation*}
   \matvec{X_-}{U_-}\matvec{X_-}{U_-}^\top\succ\Theta_{22}\succeq\matvec{\Delta_X}{\Delta_U}\matvec{\Delta_X}{\Delta_U}^\top
\end{equation*}
is satisfied due to \eqref{eiv:noise_qmi1}.
Then the analysis in~\cite{DDCQMI:BISOFFI2024_CSL} shows that the admissible system set becomes a matrix ellipsoid 
and derives an LMI equivalent based on a variant of Proposition~\ref{prop:Slem_beta} to quadratic stabilization under these assumptions.



\section{MAIN RESULT}\label{sec:main}

\subsection{Problem formulation}
We introduce \emph{data perturbation} as a noise model that includes both system noise and EIV.
\begin{definition}\label{dfn:dptb}
   \begin{subequations}
      By open-loop experiments on an unknown discrete-time LTI system, we obtain data $(X_+,X_-,U_-)$ such that
      \begin{equation}\label{stateeq_mat}
         X_+-\Delta_Z=A_s(X_--\Delta_X)+B_s(U_--\Delta_U),
      \end{equation}
      where $(\Delta_Z,\Delta_X,\Delta_U)$ is noise constrained within a linear subspace whose elements are given by
      \begin{equation}\label{lin_image}
         \Delta\coloneqq\matvecvec{\Delta_Z}{-\Delta_X}{-\Delta_U}=E\hDel
      \end{equation}
      with a known $E\inX{R}{n\times p}$ and $\hDel\inX{R}{p\times T}$.
   \end{subequations}
   We call the noise $\Delta$ data perturbation.
\end{definition}

From Definition~\ref{dfn:dptb}, we obtain
\begin{equation}\label{dptb_dfn2}
   [I_n\ A_s\ B_s]\mathbf{X}=[I_n\ A_s\ B_s]E\hDel,
\end{equation}
where $\mathbf{X}\coloneqq [X_+^\top\ -X_-^\top\ -U_-^\top]^\top$.
The data perturbation includes the system noise and the EIV.
Indeed, \eqref{stateeq_mat} coincides with its counterpart under the system noise \eqref{dist:stateeq_mat} and the EIV \eqref{eiv:stateeq_mat} with $E=[I_n\ 0]^\top\inX{R}{(2n+m)\times n}$ and $E=I_{2n+m},$ respectively.

We assume that the data perturbation satisfies a QMI
\begin{equation}\label{qmi_noise}
   \matvec{I_{p}}{\hDel^\top}^\top\hPhi\matvec{I_{p}}{\hDel^\top}\succeq 0,
\end{equation}
with $\hPhi^\top \inX{S}{p+T}$.
Similarly, the QMI \eqref{qmi_noise} coincides with the QMI on system noise \eqref{dist:qmi_noise} with $E=[I_n\ 0]^\top\inX{R}{(2n+m)\times n}$,
and encompasses the energy bound in the EIV setting \eqref{eiv:noise_qmi2}.
Additionally, we make the following assumption.
\begin{assumption}\label{asp:hPhi}
   There holds $\hPhi \in \Pi_{p,T}$.
\end{assumption}
This assumption means that the set of noise constrained by \eqref{qmi_noise} is a matrix ellipsoid.
Assumption~\ref{asp:hPhi} corresponds to $\Phi^W\in\Pi_{n,T}$ in the QMI on system noise \eqref{dist:qmi_noise},
and $\Theta\succeq 0$ in the energy bound on EIV \eqref{eiv:noise_qmi2}.

Here, we define the admissible system set $\Sigma$ under data perturbation.
We say that a system $(A,B)$ is consistent with the data $(X_+,X_-,U_-)$ 
when there exists $\hDel$ such that 
\begin{equation}\label{stateeq_mat2}
   [I_n\ A\ B]\mathbf{X}=[I_n\ A\ B]E\hDel,
\end{equation} and \eqref{qmi_noise} hold.
We define a set of the systems consistent with the data as
\begin{equation}\label{dfn_Sigma}
   \Sigma\coloneqq\left\{(A, B)\middle| \exists \hDel\inX{R}{p\times T} \st \eqref{qmi_noise}, \eqref{stateeq_mat2}\right\}.
\end{equation}
Note that $(A_s,B_s)\in\Sigma$.
We derive the data informativity for quadratic stabilization in Definition~\ref{dfn:dinfo_qstab} with~\eqref{dfn_Sigma}.

\subsection{QMI Characterization of Admissible System Set}
We characterize the admissible system set by a QMI under the data perturbation. 
In the system noise case, a simple substitution can derive the QMI characterization \eqref{dist:qmi_sys}
since the system noise can be uniquely determined when the system matrices 
and the data are given as indicated by \eqref{dist:noise_description}.
However, the data perturbation $\hDel$ cannot be uniquely determined in general because the matrix $[I\ A\ B]E$ in~\eqref{stateeq_mat2} may have a non-trivial kernel.
Therefore, extension to the data perturbation case is not straightforward. 
The following theorem, our first main result, provides a QMI characterization with data perturbation.

\begin{theorem}\label{thm:sys_set}
   Let Assumption~\ref{asp:hPhi} hold.
   For $\Sigma$ in \eqref{dfn_Sigma},
   define $\bar{\Sigma}$ as
   \begin{equation}\label{dfn_barsigma}
      \bar{\Sigma}\coloneqq\left\{(A,\ B)\middle|\IAB^\top N\IAB\succeq 0\right\},
   \end{equation}
   where $N\inX{S}{2n+m}$ is defined as
   \begin{equation}\label{dfn_N}
     N=[E\ \mathbf{X}]\hat{\Phi}[E\ \mathbf{X}]^\top.
   \end{equation}
   Then, we have $\Sigma\subseteq\bar{\Sigma}$.
   In particular, for $E_0\coloneqq [I_n\ 0]^\top\inX{R}{(2n+m)\times n}$,
   if $\im E_0\subseteq \im E$ or $\hPhi_{22}\prec 0$, then we have $\Sigma=\bar{\Sigma}$.
\end{theorem}

\begin{proof}
   First, we prepare some notations and transform the conditions defining $\Sigma$ and $\bar{\Sigma}$.
   Let $S\coloneqq [I_n\ A\ B]$, and then the equation \eqref{stateeq_mat2} can be written as 
   \begin{equation}\label{prf_stateeq}
      S\mathbf{X}=SE\hat{\Delta}.
   \end{equation}
   In a manner similar to \eqref{qmi:Z_ellip}, we rewrite the QMI \eqref{qmi_noise} as
   \begin{equation}\label{prf_expand_qmi_noise}
      \hat{Q}\hat{Q}^\top-(\hDel-\hDel_c)R\left((\hDel-\hDel_c)R\right)^\top\succeq 0 ,
   \end{equation}
   where the matrices $\hat{Q}\inX{R}{p\times p}$ and $R\inX{R}{T\times T}$ are the Cholesky decomposition of the matrices $\hPhi|\hPhi_{22}$ and $-\hPhi_{22}$
   such that $\hPhi|\hPhi_{22}=\hat{Q}\hat{Q}^\top$ and $-\hPhi_{22}=RR^\top$ respectivity,
   and $\hDel_c=-\hPhi_{12}\hPhi_{22}^\dagger\inX{R}{p\times T}$.
   By applying Lemma~\ref{lem:param} in Appendix for \eqref{prf_expand_qmi_noise}, 
   the QMI \eqref{qmi_noise} holds if and only if there exists a matrix $M_1\inX{R}{p\times T}$ such that
   \begin{equation}\label{prf_qmi_noise}
      (\hDel-\hDel_c)R=\hat{Q}M_1,\ M_1M_1^\top\preceq I_p.
   \end{equation}
   Thus, $(A,B)\in\Sigma$ if and only if 
   there exists $\hDel\inX{R}{p\times T}$ that satisfies \eqref{prf_stateeq} and \eqref{prf_qmi_noise} with some matrix $M_1\inX{R}{p\times T}$.   
   Moreover, the condition for $(A,B)\in\bar{\Sigma}$ can be transformed into
   \begin{subequations}
      {\small
      \begin{align}
        &S[E\ \mathbf{X}] \hPhi[E\ \mathbf{X}]S^\top\succeq 0\notag\\
        \Leftrightarrow&S\left(E\hat{Q}\hat{Q}^\top E^\top\!-(\mathbf{X}\!-\!E\hDel_c)RR^\top(\mathbf{X}\!-\!E\hDel_c)\right)S^\top\succeq 0\label{prf_expand_qmi1}\\
        \Leftrightarrow&SQ(SQ)^\top\!-S(\mathbf{X}-\Delta_c)R\left(S(\mathbf{X}-\Delta_c)R\right)^\top\succeq 0\label{prf_expand_qmi2},
      \end{align}}
   \end{subequations}

   \noindent
   where $Q=E\hat{Q}\inX{R}{(2n+m)\times p}$ and $\Delta_c=E\hDel\inX{R}{(2n+m)\times T}$,
   where \eqref{prf_expand_qmi1} is derived using \eqref{N_bunkai}.
   By applying Lemma~\ref{lem:param} again for \eqref{prf_expand_qmi2}, 
   $(A,B)\in\bar{\Sigma}$ if and only if there exists a matrix $M_2\inX{R}{p\times T}$ such that
   \begin{equation}\label{prf_qmi_sys}
      S(\mathbf{X}-\Delta_c)R=SQM_2,\ M_2M_2^\top\preceq I_p.
   \end{equation}

   Second, we show that $\Sigma\subseteq\bar{\Sigma}$.
   Assume that $(A,B)\in\Sigma$, i.e., there exists $\hDel\inX{R}{p\times T}$ such that \eqref{prf_stateeq} and \eqref{prf_qmi_noise} hold with some matrix $M_1\inX{R}{p\times T}$.
   We have $SE(\hDel-\hDel_c)R=S(\mathbf{X}-\Delta_c)R=SQM_1$ from \eqref{prf_stateeq} and \eqref{prf_qmi_noise}.
   Since $M_1M_1^\top\preceq I_p$, by taking $M_2=M_1$, $M_2$ satisfies \eqref{prf_qmi_sys}.
   Therefore, we obtain $(A,B)\in\bar{\Sigma}$ and $\Sigma\subseteq\bar{\Sigma}$.

   Finally, we show that $\bar{\Sigma}\subseteq\Sigma$ under $\hPhi_{22}\prec 0$ or $\im E_0\subseteq\im E$.
   Assume that $(A,B)\in\bar{\Sigma}$, i.e., there exists a matrix $M_2$ such that \eqref{prf_qmi_sys} holds.
   Then, we prove $(A,B)\in\Sigma$ by explicitly constructing $\hDel$ 
   such that \eqref{prf_stateeq} and \eqref{prf_qmi_noise} hold with some matrix $M_1\inX{R}{p\times T}$.
   When $\hPhi_{22}\prec 0$, we define $\hDel\inX{R}{p\times T}$ as
   \begin{equation}\label{def_Delta2}
      \hDel\coloneqq \hat{Q}(SQ)^\dagger S(\mathbf{X}-\Delta_c)+\hDel_c.
   \end{equation}
   We confirm that \eqref{prf_stateeq} is fulfilled with $\hDel$.
   Since $R$ is invertible due to $\hPhi_{22}\prec 0$, 
   \begin{equation}\label{prf_im}
      S(\mathbf{X}-\hDel_c) =SQM_2R^{-1}.
   \end{equation}
   By substituting \eqref{def_Delta2} into \eqref{prf_stateeq},
   we obtain
   \begin{equation*}
      \begin{split}
        SE\hDel=&SQ(SQ)^\dagger S(\mathbf{X}-\Delta_c)+S\Delta_c\\
        =&S(\mathbf{X}-\Delta_c)+S\Delta_c=S\mathbf{X},
      \end{split}
   \end{equation*}
   where we use $SQ(SQ)^\dagger S(\mathbf{X}-\Delta_c)=S(\mathbf{X}-\Delta_c)$ due to \eqref{prf_im}.
   Next, we confirm that \eqref{prf_qmi_noise} is fulfilled with $\hDel$ and $M_1\coloneqq (SQ)^\dagger SQ M_2$.
   $M_1M_1^\top\preceq I_p$ is fulfilled because $M_2M_2^\top\preceq I_p$ and $(SQ)^\dagger SQ\preceq I_p$ due to Lemma~\ref{lem:prjmatrix} in Appendix.
   By substituting \eqref{def_Delta2} into \eqref{prf_qmi_noise}, we obtain
   \begin{align*}
      (\hDel-\hDel_c)R=&\hat{Q}(SQ)^\dagger S(\mathbf{X}-\Delta_c)R\\
      =&\hat{Q}(SQ)^\dagger SQM_2\\
      =&SQM_1.
   \end{align*}
   Therefore, we obtain $(A,B)\in\Sigma$ under $\hPhi_{22}\prec 0$.

   When $\im E_0\subseteq\im E$, we define $\hDel$ as
   \begin{equation}\label{def_Delta3}
      \hDel\coloneqq \hat{Q}(SQ)^\dagger S(\mathbf{X}-\Delta_c)RR^\dagger+FS(\mathbf{X}-\Delta_c)(I-RR^\dagger)+\hDel_c,
   \end{equation}
   where $F\inX{R}{p\times n}$ is a matrix such that $E_0=EF$.
   The exsitence of $F$ is ensured by $\im E_0\subseteq\im E$.
   We confirm that \eqref{prf_stateeq} is fulfilled with $\hDel$.
   By substituting \eqref{def_Delta3} into \eqref{prf_stateeq},
   we obtain
   \begin{align*}
        SE\hDel
        =&SQ(SQ)^\dagger S(\mathbf{X}-\Delta_c)RR^\dagger\\
        &+SE_0S(\mathbf{X}-\Delta_c)(I-RR^\dagger)+S\Delta_c\\
        =&S(\mathbf{X}-\Delta_c)RR^\dagger+S(\mathbf{X}-\Delta_c)(I-RR^\dagger)+S\Delta_c\\
        =&S\mathbf{X},
   \end{align*}
   where we use $SE_0=I_n$ and $SQ(SQ)^\dagger S(\mathbf{X}-\Delta_c)R=S(\mathbf{X}-\Delta_c)R$ due to \eqref{prf_qmi_sys}.
   Next, we confirm that \eqref{prf_qmi_noise} is fulfilled with $\hDel$ and $M_1\coloneqq (SQ)^\dagger SQ M_2$.
   We can obtain \eqref{prf_qmi_noise} in the same way as in the case of $\hPhi_{22}\prec 0$.
   Therefore, we also obtain $(A,B)\in\Sigma$ under $\im E_0\subseteq\im E$.
\end{proof}

Theorem~\ref{thm:sys_set} provides a QMI characterization of $\Sigma$ by showing an inclusion relationship between $\Sigma$ and $\bar{\Sigma}$, and moreover identifies a condition such that those sets are equal.

\begin{figure*}[t]
   \centering
   \begin{minipage}[b]{0.5\textwidth}
       \centering
       \includegraphics[height=1.2in]{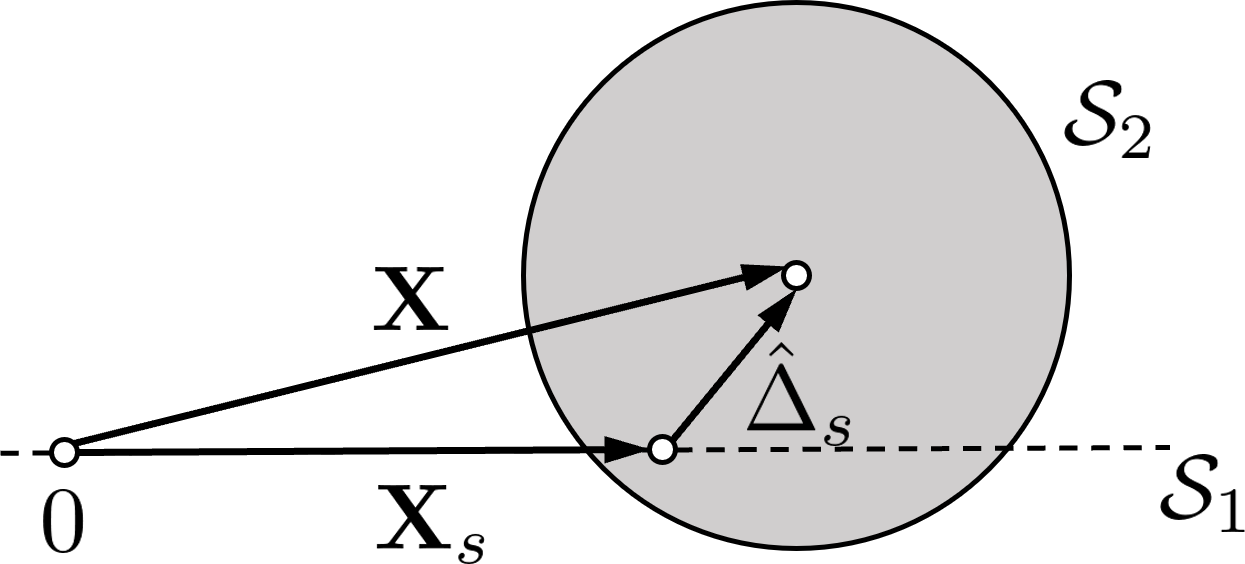}
       \subcaption{Unknown noise-free data $\mathbf{X}_s$}
       \label{fig:prj1}
   \end{minipage}%
   ~ 
   \begin{minipage}[b]{0.5\textwidth}
       \centering
       \includegraphics[height=1.2in]{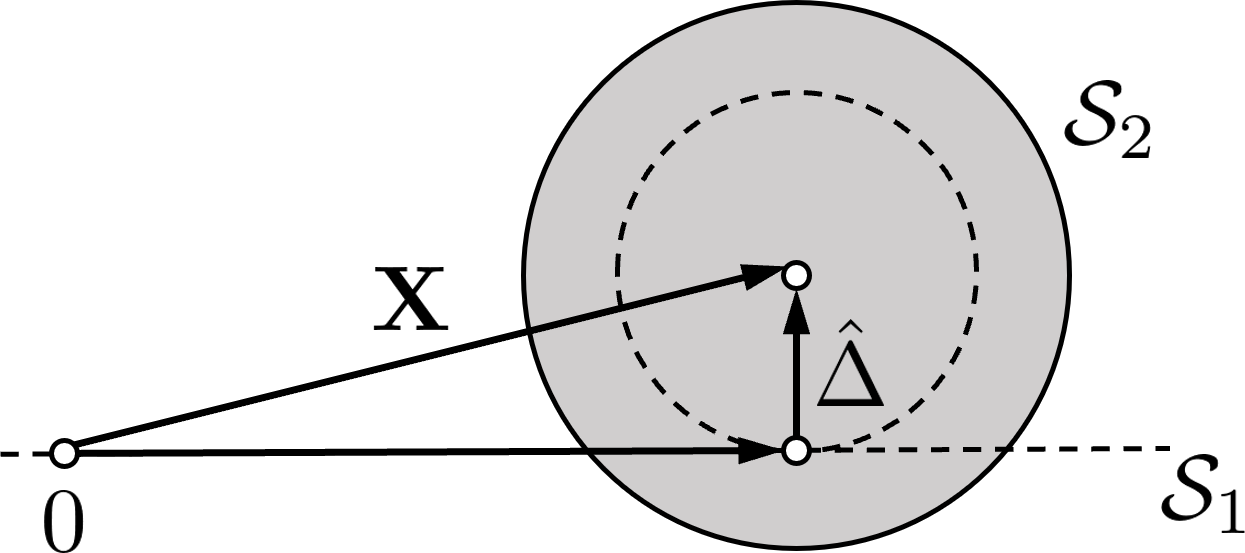}
       \subcaption{Projection of $\mathbf{X}$ onto $\mathcal{S}_2$}
       \label{fig:prj2}
   \end{minipage}
   \caption{Interpretation of $\hDel$}
   \label{fig:prj}
\end{figure*}
We provide a geometric interpretation to $\hDel$ introduced for proving $\bar{\Sigma}\subseteq \Sigma$.
We explain that $\mathbf{X}-E\hDel$ becomes the projection of the observed data $\mathbf{X}$ onto a set of the noise-free data.
We assume that $E=I_{2n+m}$, $\hDel_c=0$, and $\hat{Q}$ and $R$ are invertible for simplicity.
Then, both $\hDel$ defined as \eqref{def_Delta2} and \eqref{def_Delta3} are written as
\begin{equation}\label{dfn_Delta4}
   \hDel=\hat{Q}^{-1}(SQ)^\dagger S\mathbf{X}.
\end{equation}
Moreover, the QMI \eqref{qmi_noise} can be reduced to a norm constraint $\norm{\hat{Q}^{-1}\hDel R}_2\leq 1$, where $\norm{\cdot}_2$ is the induced 2-norm.

We denote an unknown true data perturbation as $\hDel_s$ and an unknown noise-free data as $\mathbf{X}_s\coloneqq \mathbf{X}-\hDel_s$.
Then, we obtain $\mathbf{X}_s\in\mathcal{S}_1\coloneqq \{Z\inX{R}{(2n+m)\times T}|\ SZ=0\}$ due to \eqref{prf_stateeq}
and $\mathbf{X}_s\in\mathcal{S}_2\coloneqq\{Z\inX{R}{(2n+m)\times T}| \norm{\hat{Q}^{-1}(Z-\mathbf{X})R}_2\leq 1\}$ due to $\norm{\hat{Q}^{-1}\hDel R}_2\leq 1$.
Fig~\ref{fig:prj1} illustrates a possible position of $\mathbf{X}_s$.
Since the noise-free data $\mathbf{X_s}$ is unknown, we consider the closest point in $\mathcal{S}_1$.
We project the data $\mathbf{X}$ onto the set $\mathcal{S}_1$ with the distance $d(Z_1,Z_2)\coloneqq\norm{\hat{Q}^{-1}(Z_1-Z_2)R}_2$:
\begin{equation}\label{prj}
   \begin{split}
   \min_{Z}\ &d(Z,\mathbf{X})\\
   \mathrm{s.t.}\ &Z\in\mathcal{S}_1.
   \end{split}
\end{equation}
Here, $\mathbf{X}-\hDel$ is one of the solutions of \eqref{prj}, which is illustrated in Fig.~\ref{fig:prj2}.
After all, by taking the hypothetical $\hat{\Delta}$ such that $\mathbf{X}-E\hDel$ becomes the projection onto the hyperplane corresponding to the noise-free trajectories, we derive the system set $\bar{\Sigma}$.

\subsection{Data informativity under Data Perturbation}
Quadratic stabilization is achieved when all systems in $\bar{\Sigma}$ satisfy the QMI \eqref{dist:qmi_qstab}.
However, unlike the case in Sec.~\ref{subsec:system_noise}, Proposition~\ref{prop:Slem_beta} cannot be applied 
because it requires the admissible system set to be a matrix ellipsoid.
We observe this issue and derive an extended matrix S-lemma that eliminates the need for the property.

The matrix $N$ needs to fulfill $N\in\Pi_{n,n+m}$ for applying Proposition~\ref{prop:Slem_beta}.
This condition means that the set $\bar{\Sigma}$ is as a matrix ellipsoid.
Under the system noise setting, Proposition~\ref{prop:dist_ellip} ensures that the admissible system set is a matrix ellipsoid.
However, under data perturbation setting, the condition on the admissible system set is not satisfied generally.
To explain this issue, we denote $E=[{E_+}^\top\ {E_-}^\top]^\top$ where $E_+\inX{R}{n\times p}$ and $E_-\inX{R}{(n+m)\times p}$.
Then, $N$ in \eqref{dfn_N} can be written as 
\begin{equation}
   N=\begin{bmatrix}
     E_+&X_+\\\hline E_-&-Z_-
   \end{bmatrix}
   \sqmat{\hPhi_{11}}{\hPhi_{12}}{\hPhi_{21}}{\hPhi_{22}}
   \begin{bmatrix}
      E_+&X_+\\\hline E_-&-Z_-
   \end{bmatrix}^\top,
\end{equation}
where $Z_-=[X_-^\top\ U_-^\top]^\top$.
Then, we have $N_{22}=[E_-\ -Z_-]\hPhi[E_-\ -Z_-]^\top\inX{S}{n+m}$.
In the case of system noise, $E=[I_n\ 0]^\top \inX{R}{(2n+m)\times n}$ and then
we obtain $N_{22}=Z_-\hPhi_{22}Z_-$ because $E_-=0$.
Since $\hPhi\preceq 0$ due to $\hPhi\in\Pi_{p,T}$, we have $N_{22}\preceq 0$.
However, when $E_-\neq 0$, $N_{22}\preceq 0$ cannot be guaranteed,
and thus $N\in\Pi_{n,n+m}$ does not necessarily hold.

\begin{remark}\label{rem:eiv_ellip}
   The previous study in energy bounded EIV setting~\cite{DDCQMI:BISOFFI2024_CSL} assumes \eqref{eiv:assumption}.
   Our model can be reduced to the same setting by choosing $E=I_{2n+m}$ and {\small$\hPhi=\sqmat{\Theta}{0}{0}{-I}$}, where $\Theta\succeq 0$.
   Then, we have that {\small $N_{22}=\Theta_{22}-\matvec{X_-}{U_-}\matvec{X_-}{U_-}^\top$}.
   Therefore, the assumption \eqref{eiv:assumption} implies $N_{22}\prec 0$, which is a sufficient condition for $N\in\Pi_{n,n+m}$.
\end{remark}

To overcome this issue, we derive the following theorem, our second main result, developing an extended matrix S-lemma.
\begin{theorem}[Extended matrix S-lemma]\label{thm:Slem_new}
   Let $M, N\inX{S}{q+r}$.
   Assume that $\Zqr(N)$ is nonempty and $M_{22}\preceq 0$ and $\ker M_{22}\subseteq \ker M_{12}$ hold.
   Then,
   \begin{equation}\label{slem_obj3}
     \IqZ^\top N\IqZ\succeq 0\Rightarrow\IqZ^\top M\IqZ\succ 0
   \end{equation}
   if and only if there exist scalars $\alpha\geq 0$ and $\beta>0$ such that
   \begin{equation}\label{slem_lmi3}
     M-\alpha N\succeq \sqmat{\beta I_q}{0}{0}{0}.
   \end{equation}
\end{theorem}

Compared with Proposition~\ref{prop:Slem_beta}, Theorem~\ref{thm:Slem_new} does not require $N \in \Pi_{q,r}$ but instead additionally requires $\ker M_{22} \subseteq \ker M_{12}$. 
Note that nonemptiness of $\Zqr(N)$ is a necessary condition for $N\in\Pi_{q,r}$. 
In this sense, Theorem~\ref{thm:Slem_new} requires a weaker condition for $N$, which corresponds to the admissible system set,
but a stronger condition for $M$, which corresponds to the control objective, than Proposition~\ref{prop:Slem_beta}.
The following lemma is the key to proving Theorem~\ref{thm:Slem_new}.

\begin{lemma}\label{lem:Ninfo}
   Let $M,\,N\inX{S}{q+r}$.
   Assume that $\Zqr(N)\subseteq\Zqr(M)$, $\Zqr(N)$ is nonempty, $M$ has at least one negative eigenvalue, and $M\in\Pi_{q,r}$.
   Then $N\in\Pi_{q,r}$.
\end{lemma}
\begin{proof}
   We denote $\mathcal{M}(Z)$ and $\mathcal{N}(Z)$ as
   \begin{align*}
      \mathcal{M}(Z)&\coloneqq \IqZ^\top M\IqZ\\
      &=M|M_{22}+(Z+M_{22}^\dagger M_{21})^\top M_{22}(Z+M_{22}^\dagger M_{21}),\\
      \mathcal{N}(Z)&\coloneqq \IqZ^\top N\IqZ\\
      &=N_{11}+N_{12}Z+Z^\top N_{21}+Z^\top N_{22}Z,
   \end{align*}
   respectively, where we have used \eqref{qmi:tenkai}.
   The condition $\Zqr(N)\subseteq\Zqr(M)$ is equivalent to
   \begin{equation}\label{include}
      \mathcal{N}(Z)\succeq 0\Rightarrow \mathcal{M}(Z)\succeq 0.
   \end{equation}
    Due to nonemptiness of $\Zqr(N)$, there exists a matrix $Z_0\inX{R}{r\times q}$ such that $\mathcal{N}(Z_0)\succeq 0$.
   Then, $\mathcal{M}(Z_0)\succeq 0$ also holds.

   First, we show that
   \begin{equation}\label{MN}
      x^\top M_{22}x<0\Rightarrow x^\top N_{22}x<0
   \end{equation}
   for any $x\inX{R}{r}$.
   We assume that there exists a vector $x\inX{R}{r}$ such that $x^\top M_{22}x<0$ and $x^\top N_{22}x\geq 0$, 
   and derive a contradiction.
   We define $Z'=Z_0+xy^\top$ with an arbitrary $y\inX{R}{q}$.
   Then, we have that
   \begin{align*}
      \mathcal{M}(Z')&= \mathcal{M}(Z_0)+(x^\top M_{22}x)yy^\top+\Psi xy^\top+yx^\top \Psi^\top,\\
      \mathcal{N}(Z')&= \mathcal{N}(Z_0)+(x^\top N_{22}x)yy^\top+\Omega xy^\top+yx^\top \Omega^\top,
   \end{align*}
   where $\Psi$ and $\Omega$ are denoted as $\Psi=(Z_0+M_{22}^\dagger M_{21})^\top M_{22}$ and $\Omega=N_{12}+Z_0^\top N_{22}$ respectively.
   When $\Omega x=0$, we obtain $\mathcal{M}(Z')\succeq 0$ for all $y\inX{R}{q}$ because \eqref{include} holds and $\mathcal{N}(Z')\succeq \mathcal{N}(Z_0)\succeq 0$ for all $y\inX{R}{q}$.
   However, because $x^\top M_{22}x<0$, the range of $y$ satisfying $\mathcal{M}(Z')\succeq 0$ is bounded, which contradicts the arbitrariness of $y$.
   When $\Omega x\neq 0$, we set $y$ to $y=a\Omega x$ with an arbitrary $a\geq 0$.
   Then, we obtain $\mathcal{M}(Z')\succeq 0$ for all $a\geq 0$ because \eqref{include} holds and $\mathcal{N}(Z')\succeq \mathcal{N}(Z_0)\succeq 0$ for all $a\geq 0$.
   However, the range of $a$ satisfying $\mathcal{M}(Z'\succeq 0)$ is also bounded, which contradicts the arbitrariness of $a$.
   Thus, we obtain \eqref{MN}.

   Next, we show $N_{22}\preceq 0$.
   We prove that
   \begin{equation}\label{MN0}
      x^\top M_{22}x=0\Rightarrow x^\top N_{22}x\leq0,
   \end{equation}
   for all $x\inX{R}{r}$, 
   because $M_{22}\preceq 0$, \eqref{MN} and \eqref{MN0} imply $N_{22}\preceq 0$.
   Note that $x^\top M_{22}x=0$ if and only if $M_{22}x=0$ because of $M_{22}\preceq 0$.
   We assume that there exists $x_1\inX{R}{r}$ such that $M_{22}x_1=0$ and $x_1^\top N_{22}x_1> 0$, 
   and derive a contradiction.
   Since $M_{22}$ has a negative eigenvalue, there exists $x_2$ such that $x_2^\top M_{22}x_2<0$.
   Because $x_1^\top N_{22}x_1> 0$ is assumed and $x^\top N_{22} x$ is continuous in $x$,  
   there exists a scalar $\epsilon\in\mathbb{R}$ such that $(x_1+\epsilon x_2)^\top N_{22}(x_1+\epsilon x_2)>0$.
   However, we obtain $(x_1+\epsilon x_2)^\top N_{22}(x_1+\epsilon x_2)<0$ because \eqref{MN} holds and $(x_1+\epsilon x_2)^\top M_{22}(x_1+\epsilon x_2)=\epsilon^2x_2^\top M_{22}x_2<0$.
   This is a contradiction and we obtain \eqref{MN0}.
   Therefore, we obtain $N_{22}\preceq 0$. 

   Next, we show $\ker N_{22}\subseteq \ker N_{12}$.
   This condition is obvious when $N_{22}$ is non-singular. 
   Thus, we consider the case where $N_{22}$ is singular.
   We assume that there exists $x_1\inX{R}{r}$ such that $N_{22}x_1=0$ and $N_{12}x_1\neq 0$, 
   and derive a contradiction.
   Note that $M_{22}x_1=0$ due to $M_{22}\preceq 0$ and the contrapositive of \eqref{MN}.
   Since $M_{22}$ has a negative eigenvalue, there exists $x_2$ such that $x_2^\top M_{22}x_2<0$.
   We define $Z'$ as $Z'=Z_0+x_1y_1^\top+x_2y_2^\top$ with arbitrary $y_1,y_2\inX{R}{q}$.
   Then, we have that
   \begin{align*}
      \mathcal{M}(Z')=& \mathcal{M}(Z_0)+(x_1y_1^\top+x_2y_2^\top)^\top M_{22}(x_1y_1^\top+x_2y_2^\top)\\
      &+\Psi(x_1y_1^\top+x_2y_2^\top)+(x_1y_1^\top+x_2y_2^\top)^\top \Psi^\top\\
      =&\mathcal{M}(Z_0)+(x_2^\top M_{22}x_2)y_2y_2^\top+\Psi x_2y_2^\top+y_2x_2^\top \Psi^\top,\\
      \mathcal{N}(Z')=&\mathcal{N}(Z_0)+(y_1x_1^\top+y_2x_2^\top)N_{22}(x_1y_1^\top+x_2y_2^\top)\\
      &+\Omega(x_1y_1^\top+x_2y_2^\top)+(x_1y_1^\top+x_2y_2^\top)^\top \Omega^\top\\
      =&\mathcal{N}(Z_0)+N_{12}x_1y_1^\top+y_1x_1^\top N_{21}\\
      &+py_2y_2^\top+\Omega x_2y_2^\top+y_2x_2^\top \Omega^\top,
   \end{align*}
   where $p$ is $p=x_2^\top N_{22}x_2$. 
   Note that $p<0$ due to \eqref{MN} and $x_2^\top M_{22}x_2<0$.
   We set $y_1$ to $y_1=ay$ with $y=N_{12}x_1$ and an arbitrary $a\geq 0$, and $y_2$ to $y_2=-\frac{1}{p}\Omega x_2+b y$ with an arbitrary $b\in\mathbb{R}$.
   Then, we have that
   \begin{equation*}
      \mathcal{N}(Z')=\mathcal{N}(Z_0)+(2a+pb^2)yy^\top -\frac{1}{p}(\Omega x_2)(\Omega x_2)^\top.
   \end{equation*}
   Owing to arbitrariness of $a\geq 0$, for all $b$, there exists $a\geq 0$ such that $(2a+pb^2)\geq 0$,
   and then $\mathcal{N}(Z')\succeq \mathcal{N}(Z_0)\succeq 0$ holds.
   Due to \eqref{include}, for all $b$, there exists $a\geq 0$ such that $\mathcal{M}(Z')\succeq 0$.
   However, the range of $b$ satisfying $\mathcal{M}(Z')\succeq 0$ is bounded because $x_2^\top M_{22}x_2<0$, which contradicts the arbitrariness of $b$.
   Therefore, we obtain $\ker N_{22}\subseteq \ker N_{12}$.

   Finally, we show $N|N_{22}\succeq 0$.
   $\ker N_{22}\subseteq \ker N_{12}$ enables us to write 
   \begin{equation*}
      \mathcal{N}(Z)=N|N_{22}+(Z+N_{22}^\dagger N_{21})^\top N_{22}(Z+N_{22}^\dagger N_{21})
   \end{equation*}
   by utilizing \eqref{qmi:tenkai}.
   Since $N_{22}\preceq 0$ and $\mathcal{N}(Z_0)\succeq 0$, we obtain $N|N_{22}\succeq 0$.
\end{proof}

Lemma \ref{lem:Ninfo} states that if $\Zqr(M)$ is a matrix ellipsoid and contains $\Zqr(N)$, $\Zqr(N)$ is also a matrix ellipsoid.
Next, we prove Theorem~\ref{thm:Slem_new} by leveraging Lemma~\ref{lem:Ninfo}.

\begin{proof}
   If there exist scalars $\alpha\geq 0$ and $\beta>0$ such that \eqref{slem_lmi3},
   we obtain
   \begin{equation*}
      \IqZ^\top M\IqZ\succeq \alpha\IqZ^\top N\IqZ +\beta I_q,
   \end{equation*}
   which implies \eqref{slem_obj3}.

   Next, we prove the converse.
   We assume \eqref{slem_obj3} holds.
   Since $\Zqr(N)$ is nonempty, there exists $Z_0\in\Zqr(N)$.
   By utilizing \eqref{qmi:tenkai} and \eqref{slem_obj3}, we have that
   \begin{equation}\label{prf_Msucc0}
      M|M_{22}\succeq \matvec{I_q}{Z_0}^\top M\matvec{I_q}{Z_0}\succ 0.
   \end{equation}

   When $M_{22}=0$, $M_{12}=0$ is obtained by $\ker M_{22}\subseteq \ker M_{12}$, and then $M_{11}=M|M_{22}\succ 0$ also holds.
   Thus, $\alpha=0$ and $0<\beta<\lambda_{\mathrm{min}}$ satisfy \eqref{slem_lmi3} where $\lambda_{\mathrm{min}}$ is the minimum eigenvalue of $M_{11}$.

   When $M_{22}\neq 0$, $M_{22}$ has at least one negative eigenvalue.
   Since the preconditions and \eqref{prf_Msucc0} lead to $M\in\Pi_{q,r}$, we can apply Lemma~\ref{lem:Ninfo}, and then we have $N\in\Pi_{q,r}$.
   Thus, we obtain \eqref{slem_lmi3} by applying Proposition~\ref{prop:Slem_beta}.
\end{proof}

We obtain an LMI condition equivalent to the data informativity by applying Theorem~\ref{thm:Slem_new} for $M$ defined as \eqref{dist:dfn_M} and $N$ defined as \eqref{dfn_N}.
\begin{theorem}\label{thm:result}
   Let Assumption~\ref{asp:hPhi} hold.
   Assume that the data $(X_+,X_-,U_-)$ and the data perturbation $\Delta$ satisfy \eqref{stateeq_mat} and the data perturbation satisfies \eqref{lin_image} and \eqref{qmi_noise}.
   Then, the data $(X_+,X_-,U_-)$ is informative for quadratic stabilization
   if there exists an $n$-dimensional positive definite matrix $P\succ 0$, a matrix $L\inX{R}{m\times n}$, and a positive scalar $\beta>0$ such that
   \begin{equation}\label{LMI_result}
     \begin{bmatrix}
       P-\beta I_n&0&0&0\\
       0&-P&-L^\top &0\\
       0&-L&0&L\\
       0&0&L^\top &P
     \end{bmatrix}-\sqmat{N}{0}{0}{0}\succeq 0,
   \end{equation}
   where $N$ is defined as \eqref{dfn_N}.
   In particular, when $\hPhi_{22}\prec 0$ or $\im E_0\subseteq\im E$, \eqref{LMI_result} is a necessary and sufficient condition for data informativity for quadratic stabilization.
   Moreover, if \eqref{LMI_result} is feasible, a controller $K=LP^{-1}\inX{R}{m\times n}$ stabilizes all systems in $\Sigma$.
\end{theorem}
\begin{proof}
   We have that $\Sigma\subseteq \bar{\Sigma}$ by Theorem~\ref{thm:sys_set},
   where $\bar{\Sigma}$ is characterized by the QMI with $N$.
   $\bar{\Sigma}$ is not empty because it contains the true system $(A_s,B_s)$.
   Thus, $\Zab{n}{n+m}(N)$ is not empty.
   On the other hand, the matrix $M$ defined as \eqref{dist:dfn_M} satisfy $M_{22}\preceq 0$ and $\ker M_{22}\subseteq\ker M_{12}$.
   Thus, we can apply Theorem~\ref{thm:Slem_new} for $M$ and $N$ and obtain a following statement:
   All systems in $\bar{\Sigma}$ satisfy the common Lyapunov inequality \eqref{dist:qstab} with a common controller if and only if 
   there exists a controller $K$, a positive definite matrix $P$, and scalars $\alpha\geq 0$ and $\beta>0$ such that
   \begin{equation}\label{prf_lmi}
      M-\alpha N-\sqmat{\beta I_n}{0}{0}{0}\succeq 0.
   \end{equation}
   Since $\Sigma\subseteq \bar{\Sigma}$, \eqref{prf_lmi} is sufficient condition for the data informativity for quadratic stabilization.
   We transform \eqref{prf_lmi} into an LMI condition based on the change of variable $L=KP$ and Schur complement.
   By scaling with respect to $\alpha$, we obtain the LMI condition \eqref{LMI_result} (for more details, see the proof of \cite[Theorem 5.1(a)]{DDCQMI:Waarde2023_siam_qmi})

   In particular, when $\hPhi_{22}\prec 0$ or $\im E_0\subseteq\im E$,
   we obtain $\Sigma=\bar{\Sigma}$ from Theorem~\ref{thm:sys_set}, 
   and then, the feasibility of the LMI condition \eqref{LMI_result} is equivalent to the data informativity for quadratic stabilization.
\end{proof}

In the case of the system noise or EIV, 
Theorem~\ref{thm:result} provides a necessary and sufficient condition for the data informativity for quadratic stabilization,
because $\im E_0\subseteq\im E$.
The LMI condition \eqref{LMI_result} coincides with that in Proposition~\ref{prop:dist_result} where $E=[I_n\ 0]^\top\inX{R}{(2n+m)\times n}$
and is relaxed condition of the result in EIV setting~\cite[Theorem 1]{DDCQMI:BISOFFI2024_CSL} where $E=I_{2n+m}$.


\section{CONCLUSION AND FUTURE WORK}\label{sec:conclusion}
In this study, we have introduced the notion of data perturbation as a general noise model including system noise and EIV,
and derived an LMI condition equivalent to the data informativity for quadratic stabilization under the data perturbation constrained by a QMI.
We have derived the QMI characterization of the admissible system set with the geometric interpretation,
and obtained the LMI condition by applying a newly developed matrix S-lemma.
Our analysis provides a unified result of the existing ones under the system noise and EIV without any restrictive assumptions.
Future work includes assessing the data informativity for other analysis and control problems under data perturbation.


\section*{Appendix}
We introduce lemmas used in the proof of Theorem~\ref{thm:sys_set}.
\begin{lemma}\label{lem:param}(\cite[Lemma A.1]{DDCQMI:Waarde2023_siam_qmi})
  For $A\inX{R}{q\times p}$ and $B\inX{R}{r\times p}$,
  $A^\top A\succeq B^\top B$ holds if and only if there exists a matrix $M\inX{R}{r\times q}$ such that
  $B=MA,\ M^\top M\preceq I_q.$
\end{lemma}
\begin{lemma}\label{lem:prjmatrix}
  For $A\inX{R}{m\times n}$, $A^\dagger A\preceq I_n$ holds.
\end{lemma}
\begin{proof}
  This lemma follows the fact that the eigenvalues of $A^\dagger A$ are contained by a set $\{0,\ 1\}$ \cite[Prop. 6.1.6]{bernstein2009matrix}.
\end{proof}

\bibliographystyle{IEEEtran}
\bibliography{IEEEabrv,Robust_DDC,DDC,mathematics}
\end{document}